\documentclass[
reqno]{amsart}

\usepackage{amsmath,amssymb,eufrak}
\usepackage[all]{xy}
\usepackage[active]{srcltx} 
\usepackage{mathrsfs}

\newtheorem{thm}{Theorem}[section]
\newtheorem{lm}[thm]{Lemma}
\newtheorem{co}[thm]{Corollary}
\newtheorem{pr}[thm]{Proposition}

\theoremstyle{definition}

\newtheorem{df}[thm]{Definition}

\newtheorem{exm}[thm]{Example}

\newtheorem{rem}[thm]{Remark}

\numberwithin{equation}{section}

\newcommand*{\spn}{\mathrm{span}}

\newcommand{\ptn}{\mathbin{\widehat{\otimes}}}

\newcommand*{\ptens}[1]{\mathop{\widehat\otimes}_{#1}}

\newcommand{\sgn}{\mathop{\mathrm{sgn}}\nolimits}

\newcommand*{\lar}{\leftarrow}

\newcommand{\CC}{\mathbb{C}}

\newcommand{\R}{\mathbb{R}}
\newcommand{\N}{\mathbb{N}}
\newcommand{\Z}{\mathbb{Z}}

\newcommand{\cO}{\mathcal{O}}

\newcommand{\fg}{\mathfrak{g}}
\newcommand*{\fv}{\mathfrak{v}}

\renewcommand{\le}{\leqslant}
\renewcommand{\ge}{\geqslant}

\let \al         =\alpha
\let \be         =\beta
\let \ga         =\gamma
\let \de         =\delta

\let \te         =\theta
\let \io         =\iota

\let \la         =\lambda
\let \si         =\sigma
\let \up         =\upsilon
\let \om         =\omega

\let \phi         =\varphi

\title[Nilpotent Lie algebras and homological epimorphisms]
{Arens-Michael envelopes of nilpotent Lie algebras, holomorphic
functions of exponential type, and homological epimorphisms}
\author{O. Yu. Aristov}
\thanks{This work was supported by the RFBR grant no. 19-01-00447.}
\email{aristovoyu@inbox.ru}

\begin{document}
\begin{abstract} 
Our aim is to give an explicit description of the Arens-Michael
envelope for the universal enveloping algebra of a
finite-dimensional nilpotent complex Lie algebra. It turns out
that the Arens-Michael envelope belongs to a class of completions
introduced by R.~Goodman in 70s.  To find a precise form of this
algebra we preliminary characterize the set of holomorphic
functions of exponential type on a simply connected nilpotent
complex Lie group. This approach leads to unexpected connections
to Riemannian geometry and the theory of order and type for
entire functions.

As a corollary, it is shown that the Arens-Michael envelope
considered above is a homological epimorphism. So we get a
positive answer to a question investigated earlier by Dosi and
Pirkovskii.
\end{abstract}
 \maketitle

 \section{Introduction}
A research program on Arens-Michael envelopes and homological
epimorphisms was initiated by Joseph Taylor in his seminal papers
\cite{T1,T2}.  Taylor was inspired by his previous results on
multi-operator holomorphic functional  calculus and some
consideration that can be incorporated in those circle of ideas
that is called Noncommutative Geometry nowadays. Homological
epimorphisms play important role in modern attempts to find
generalizations of Taylor spectrum and functional calculus for
noncommutative algebras (see \cite{Do10A,Do10B}).

We pursue two aims: to get an explicit description of the
Arens-Michael envelope for $U(\fg)$ and prove that it is a
homological epimorphism,  both for a  finite-dimensional nilpotent
complex Lie algebra $\fg$.  (Here $U(\fg)$ denotes the universal
enveloping algebra of $\fg$.)

\subsection*{Arens-Michael envelopes}
In this text, all vector spaces and algebras are considered over
the field $\CC$ of complex numbers. All algebras and their
homomorphisms are assumed to be unital.

Recall that  a complete Hausdorff locally convex topological
algebra with jointly continuous multiplication is called a
\emph{$\ptn$-algebra}. (Here $\ptn$ is the sign for the complete
projective tensor product of locally convex spaces.)  A~$\ptn$-
algebra A is called an \emph{Arens-Michael algebra} (or a
\emph{complete m-convex algebra}) if its topology is determined by
a family of submultiplicative prenorms $(\|\cdot\|_\al)$ (i.e.,
$\|ab\|_\al \le\|a\|_\al\|b\|_\al$ for all $a, b\in A$).

The \emph{Arens-Michael envelope} of a $\ptn$-algebra $A$
\cite[Chap.~5]{X2} is  a pair $(\widehat A, \io_A)$, where
$\widehat A$ is an Arens-Michael algebra  and $\io_A$ is a
continuous homomorphism $A \to \widehat A$ s.t. for any
Arens-Michael algebra $B$ and for each continuous homomorphism
$\phi\!: A \to B$ there exists a unique continuous homomorphism
$\widehat\phi\!:\widehat A \to B$ making the following diagram
commutative
\begin{equation}\label{AMen}
  \xymatrix{
A \ar[r]^{\io_A}\ar[rd]_{\phi}&\widehat A\ar@{-->}[d]^{\widehat\phi}\\
 &B\\
 }
\end{equation}
Note that it suffices to consider only homomorphisms with values
in Banach algebras. The Arens-Michael envelope always exists and
is unique up to a topological isomorphism. The algebra $\widehat
A$ is the completion of $A$ w.r.t. the family of all continuous
submultiplicative prenorms. (An  arbitrary associative
$\CC$-algebra can be considered as a $\ptn$-algebra w.r.t. the
strongest locally convex topology; in this case, all prenorms are
automatically continuous.)

To formulate our first main result recall some terminology and
notation from \cite{Go78,Go79,Pir_stbflat}.  Consider a
finite-dimensional nilpotent complex Lie algebra~$\fg$ and fix a
positive decreasing filtration $\mathscr{F}$ on~$\fg$, i.e., a
decreasing chain of subspaces
\begin{equation}\label{posfilt}
\fg=\fg_1 \supset \fg_2 \supset \cdots  \supset\fg_k \supset\fg_{k+1}=0,
\qquad [\fg_i,\fg_j]\subset \fg_{i+j}\,.
\end{equation}
Consider a basis $(e_1,\ldots, e_m)$ in $\fg$ and set
\begin{equation}\label{widef}
w_i\!:=\max\{j:\,e_i\in \fg_j \}\quad\text{and}\quad
w(\alpha)\!:=\sum_i w_i\alpha_i \,,
\end{equation}
where $\alpha=(\alpha_1,\ldots,\alpha_m)\in\Z_+^m$. In the
following  we assume that $(e_i)$ is  an
\emph{$\mathscr{F}$-basis}, i.e., $w_i\le w_{i+1}$  for all $i$,
and $\fg_j = \spn\{e_i : w_i \ge j\}$ for all $j$. A sequence
${\mathscr M}=\{ M_\alpha: \alpha\in\Z_+^m\}$ of positive numbers
is called an \emph{${\mathscr F}$-weight sequence} if $M_0=1$ and
$M_\gamma\le M_\alpha M_\beta$ whenever $w(\gamma)\ge
w(\alpha)+w(\beta)$. Given an ${\mathscr F}$-weight sequence
${\mathscr M}$, consider the space
\begin{multline}\label{UfgMde}
U(\fg)_{\mathscr M}\!:= \Bigl\{ x=\sum_{\alpha\in\Z_+^m} c_\alpha
e^\alpha\in [U] \!:\\
 \| x\|_r\!:=\sum_\alpha |c_\alpha|
\alpha!\, M_\alpha r^{w(\alpha)}<\infty \;\forall r>0\Bigr\}\,,
\end{multline}
 where $(e^\al\,:\al\in\Z_+^m)$ is the PBW-basis in
$U(\fg)$ associated with  $(e_i)$, and $[U]$ is the set of formal
power series w.r.t. $(e^\al)$. It is proved in
\cite[Th.~6.3]{Go78} that the multiplication on  $U(\fg)$ extends
to $U(\fg)_{\mathscr M}$ and $U(\fg)_{\mathscr M}$ is a
$\ptn$-algebra.

The standard choice for $\mathscr{F}$ is  the lower central
series of $\fg$ that is defined inductively by $\fg_i\!:=[\fg,
\fg_{i-1}]$. Consider the $\mathscr F$-weight sequence $\mathscr
M_1$ defined by $ M_\al\!:=w(\al)^{-w(\al)}$. (This is $\mathscr
M_p$ as defined in \cite[Sec.~2, Exm.~1]{Go79} with $p=1$.)

\begin{thm}\label{AMEdescrnil}
Let $\fg$ be a finite-dimensional nilpotent complex Lie algebra,
$\mathscr{F}$  the lower central series, and $(e_i)$ an
$\mathscr{F}$-basis.  Then the topology on the $\ptn$-algebra
$$
U(\fg)_{\mathscr M_1}= \Bigl\{ x=\sum_{\alpha\in\Z_+^m} c_\alpha
e^\alpha\in [U]\! : \| x\|_r=\sum_\alpha |c_\alpha| \alpha!\,
w(\al)^{-w(\al)} r^{w(\alpha)}<\infty \;\forall r>0\Bigr\}
$$
can be determined by system of submultiplicative prenorms, i.e.,
it is an Arens-Michael algebra. Moreover, the natural
homomorphism $U(\fg)\to U(\fg)_{\mathscr M_1}$ is an
Arens-Michael envelope.
\end{thm}

The first step of the proof of Theorem~\ref{AMEdescrnil} is
reduction $\widehat U(\fg)$ to the Arens-Michael envelope of the
algebra $\mathscr{A}(G)$ of analytic functionals on the
corresponding simply connected complex Lie group $G$
(Proposition~\ref{csicHFG}). Further, we use  the identification
(obtained by Akbarov in \cite{Ak08}) between the strong dual
space of $\widehat{\mathscr{A}}(G)$  and the locally convex space
$\cO_{exp}(G)$ of holomorphic functions of exponential type
on~$G$.

The key technical result is Theorem~\ref{eelUsi}, which gives
estimations for growth rate of a word length function. It is not
particularly  original: an explicit formulation and a part of a
proof can be found in \cite[II.4.17]{DER}, where the statement is
given for a right invariant Riemannian distance. The reasoning is
essentially contained in the proofs of \cite[Pr.~IV.5.6 and
IV.5.7]{VSC92} but in a latent form.  Moreover, the main goal of
[ibid.] is to rate volume growth; so the reader needs some
additional work to extract an argument for distances. Very close
results are contained in \cite[Th.~4.2]{Kar94} and
\cite[Pr.~7.25]{Be96} but in variations that are not completely
satisfactory for our purposes. Besides, a careful examination
shows that estimates for Riemannian distances are based on
estimates for length functions; thus there is a direct way to
establish Theorem~\ref{eelUsi}, which passes Riemannian geometry.
So I include a complete proof in Appendix, where a connection
with Riemannian distances is also explained.

From Theorem~\ref{eelUsi} we obtain an explicit description of
$\cO_{exp}(G)$ (Theorem~\ref{exptypdesc}) and show that the
functions of exponential type forms exactly the dual space of
$U(\fg)_{\mathscr{M}_1}$, which implies our assertion.
(Theorem~\ref{exptypdesc}, which plays only supporting role in our
argument, is of independent interest itself. This result is an
essential part of the description of the space of holomorphic
functions of exponential type on an arbitrary connected complex
Lie group --- the subject that is discussed in \cite{Ar19}.)

\subsection*{Homological epimorphisms}
Let $A$ be a $\ptn$-algebra. Recall that an
\emph{$A$-$\ptn$-bimodule} is a complete Hausdorff locally convex
space endowed with a structure of a unital  $A$-bimodule s.t. both
left and right multiplications are jointly continuous. Below
$\ptens{A}$ denotes the projective tensor product of
$A$-$\ptn$-modules.

A homomorphism of $\ptn$-algebras $A\to B$ is called
a~\emph{homological epimorphism} if the induced functor between
the bounded derived categories of $\ptn$-modules is fully
faithful. This condition is equivalent to the following:  for some
(or what is the same, for each) admissible projective resolution
$0 \lar A \lar L_\bullet$ in the relative category of
$A$-$\ptn$-bimodules the
 complex
\begin{equation}\label{nqflres}
0 \lar B\ptens{A} B  \lar B\ptens{A} L_0\ptens{A} B   \lar \cdots
 \lar B\ptens{A} L_{n}\ptens{A} B  \lar\cdots
\end{equation}
is admissible \cite[Rem.~6.4]{Pir_qfree}.  We follow the
terminology from [ibid.]; the alternative terminologies: $A\to B$
is a \emph{localization} or $B$ is \emph{stably flat over $A$} are
used in \cite{Pir_stbflat}. Definitions of homological notions for
$\ptn$-algebras can be found in \cite{X1,X2,X_HOA}. We do not need
details here because the only necessary fact on homological
epimorphisms is Theorem~\ref{PirkThUM} below.

Taylor \cite{T2} proved that the Arens-Michael envelope is a
homological epimorphism for the polynomial algebra in $n$
generators and for the free algebra in $n$ generators.  The
natural next step is to consider the universal enveloping algebra
$U(\fg)$ of a finite-dimensional complex Lie algebra $\fg$. The
first and a bit disappointing result in this direction, also due
to Taylor \cite{T2}, asserts that if $\fg$ is semisimple, then the
Arens-Michael envelope $\io_U\!:U(\fg)\to \widehat U(\fg)$ fails
to be a homological epimorphism, in contrast to the abelian case.
Many years later the results for solvable $\fg$ began to appear.
Dosiev \cite[Th.~10]{Do03} proved that $\io_U$ is a homological
epimorphism provided $\fg$ is metabelian (i.e., $[\fg, [\fg, \fg]]
= 0$). Later, Pirkovkii generalized this result to positively
graded Lie algebras \cite[Th.~6.19]{Pir_stbflat} and Dosiev
generalized it to nilpotent Lie algebras satisfying a condition of
''normal growth'' \cite{Do09}. A natural conjecture became that
the same is true for each nilpotent Lie algebra $\fg$. On the
other hand, it was shown in \cite{Pi4} that  $\io_U$   can be a
homological epimorphism only when $\fg$ is solvable.

Another approach was introduced Pirkovkii in \cite{Pir_qfree}. An
Ore extension iteration gives sometimes a direct construction of
$\widehat U(\fg)$  and a method to prove that $\io_U$ is a
homological epimorphism. This method requires some technical work;
nonetheless, it gives the only known up to date example of a
solvable non-nilpotent~$\fg$ s.t. $\io_U$ is a homological
epimorphism, namely, the two-dimensional solvable non-abelian Lie
algebra.

Now we formulate our second main result.
\begin{thm}\label{thmmail}
Let $\fg$ be a finite-dimensional nilpotent complex Lie algebra.
Then the Arens-Michael envelope $\io_U\!:U(\fg)\to \widehat
U(\fg)$  is a   homological epimorphism.
\end{thm}
The  proof is based on Theorem~\ref{AMEdescrnil} and the
following Pirkovskii's theorem. (The definition of an entire
${\mathscr F}$-weight sequence see below in~\eqref{entire}.)
\begin{thm} \label{PirkThUM}
\cite[Th.~7.3]{Pir_stbflat} Let $\fg$ be a finite dimensional
nilpotent complex Lie algebra,  and let ${\mathscr M}$ be an
entire ${\mathscr F}$-weight sequence for some positive
filtration ${\mathscr F}$. Then the natural homomorphism
$U(\fg)\to U(\fg)_{\mathscr M}$ is a   homological epimorphism.
\end{thm}

The present paper is organized as follows. In section~\ref{PR},
some preliminary results on the algebra of analytic functionals,
submultiplicative weights, and holomorphic functions of
exponential type are collected.  In section~\ref{RM}, a result on
growth of word length functions (Theorem~\ref{eelUsi}) is
formulated  and applied to characterize  holomorphic functions of
exponential type on a simply connected nilpotent  Lie group
(Theorem~\ref{exptypdesc}). The proofs of
Theorems~\ref{AMEdescrnil},~\ref{thmmail}  and some application
are contained in  section~\ref{PMR}. Appendix includes the proof
of Theorem~\ref{eelUsi} and the explanation how this assertion is
connected with growth rate of invariant Riemannian distances.

\subsection*{Acknowledgements} The author thanks S.~Akbarov for
useful comments.

\section{Preliminary results}\label{PR}
\subsection*{Reduction to the algebra of analytic functionals}

Let $G$ be a complex Lie group with Lie algebra $\fg$.  Consider
the Fr\'{e}chet algebra $\cO(G)$ of holomorphic functions on $G$
and its strong dual space $\cO(G)'$ endowed with the convolution
multiplication.  This $\ptn$-algebra is denoted by
${\mathscr{A}}(G)$ and is called the \emph{algebra of analytic
functionals} on $G$ (for a general complex Lie group,
${\mathscr{A}}(G)$ is introduced by G.~Litvinov in \cite{Lit}).

Consider elements of $U(\fg)$ as left-invariant differential
operators on $\cO(G)$ and define the homomorphism $\tau$ by
\begin{equation}\label{taudef}
\tau\!:U(\fg) \to \mathscr{A}(G)\!:\langle \tau(X), f\rangle\!:= \langle
\de_e, Xf\rangle \qquad (X\in U(\fg),\, f \in  \cO(G))\,,
\end{equation}
where $\de_e$ is denoted the delta-function at $e$.

Recall that the Arens-Michael envelope is a functor from the
category of $\ptn$-algebras to the category of  Arens-Michael
algebras. For any  $\ptn$-algebra homomorphism $\te$ we denote
the corresponding homomorphism between the Arens-Michael
envelopes by $\widehat \te$.

\begin{pr}\label{csicHFG}
Let $G$ be a simply connected complex Lie group with Lie
algebra~$\fg$. Then $\widehat\tau\!:\widehat{U}(\fg) \to
\widehat{\mathscr{A}}(G)$ is an Arens-Michael algebra
isomorphism.
\end{pr}
\begin{rem}
The similar result is valid for real Lie groups.  If $G$ is a
simply connected real Lie group and  $\fg_\CC$ denotes the
complexification of its Lie algebra, then   $U(\fg_\CC)$ and the
algebra $\mathcal{E}'(G)$ of compactly supported distribution
have  the  same Arens-Michael envelope  \cite[p.~250]{T2}.
\end{rem}
To prove Proposition~\ref{csicHFG} we need the following lemma.
\begin{lm}\label{BwhA}
Let $\te\!:A\to B$ be an epimorphism of $\ptn$-algebras. Suppose
that  there exists a continuous homomorphism $j\!:B\to
\widehat{A}$ s.t. $\io_{A}=j\te$.   Then
$\widehat{\te}\!:\widehat{A}\to \widehat{B}$ is an Arens-Michael
algebra isomorphism.
\end{lm}
\begin{proof}
The homomorphisms $\io_{A}$ and $\io_{B}$ have dense ranges; so
they are epimorphisms. Since, by assumption, $\te$ is an
epimorphism, it follows from $\widehat{\te}\io_{A}=\io_{B} \te$
that  $\widehat{\te}$ also is  an epimorphism.

By the universal the property of the Arens-Michael envelope, there
is a continuous homomorphism $\al\!:\widehat{B}\to \widehat{A}$
s.t. $j=\al\io_{B}$. Therefore,
$$
\al \widehat{\te}\io_{A}=\al\io_{B} \te=j\te=\io_{A}\,.
$$
Since $\io_{A}$  is an epimorphism, we have $\al
\widehat{\te}=1$.   Therefore, $\widehat{\te}\al
\widehat{\te}=\widehat{\te}$. But $\widehat{\te}$  is an
epimorphism so  $\widehat{\te}\al =1$. Thus $\widehat{\te}$ is
invertible.
\end{proof}

\begin{proof}[Proof of Proposition~\ref{csicHFG}]
Since $G$ is simply connected, it follows from
\cite[Pr.~9.1]{Pir_stbflat} that there exists a unique continuous
homomorphism $j\!:{\mathscr{A}}(G)\to \widehat{U}(\fg)$ s.t.
$\io_{U(\fg)}=j\tau$. On the other hand, $\tau$ has dense range
provided  $G$ is connected (because the dual map is injective;
see, e.g., discussion after formula~(42) in [ibid.]). So we can
apply Lemma~\ref{BwhA}.
\end{proof}

\subsection*{Submultiplicative weights and length functions}

\begin{df}
(A) A \emph{submultiplicative weight} on a locally compact group $G$ is a
non-negative locally bounded function $\om\!: G \to \R$ s.t.
$$
\om(gh)\le \om(g)\om(h)\qquad (g, h \in G)\,.
$$

(B)  A \emph{length function} on a locally compact group $G$ is a locally
bounded function $\ell\!:G\to \R$ s.t.
$$
 \ell(gh)\le \ell(g)+\ell(h)\qquad (g, h \in G)\,.
$$
\end{df}

It is not hard to check that a strictly positive submultiplicative
weight maps $G$ to $[1,+\infty)$, and a length function  maps $G$
to $[0,+\infty)$.

\begin{rem}
(A) We accept the terminology from \cite{Wi13}. As a rule, $\ell$
(or $\om$) is assumed to be symmetric, i.e., $\ell(e) = 0$ and
$\ell(g^{-1})=\ell(g)$ (or $\om(e) = 1$ and $\om(g^{-1})=\om(g)$);
see Appendix. The short term 'weight' (see, e.g.,
\cite{Sch93,Da00}) is usual but it has too many different senses.
Akbarov in \cite{Ak08} employs 'semicharacter' for
'submultiplicative weight'. L.~Schweitzer in \cite{Sch93} prefers
'gauge' for 'length function'.  The words  'seminorm' for
'submultiplicative weight' \cite[Sect.~4.2.2]{War72} and 'modulus'
for 'length function' \cite{DER} can be also  used in the Lie
group context.

(B) We assume that $\om$ or $\ell$ is locally bounded because we
follow \cite{Ak08} principally. The more common admissions that
$\om$ or $\ell$ is measurable (w.r.t. the Haar measure) or Borel
are stronger. Indeed, in these cases, submultiplicativity or
subadditivity   implies that $\om$ or $\ell$, resp., is bounded
on compact sets and, hence, is locally bounded (see
\cite[Pr.~2.1]{Dz86} and \cite[Th.~1.2.11]{Sch93}).

(C)  The  map $\ell\mapsto (\om(g)\!:=e^{\ell(g)})$ is a
bijection between the set of length functions and the set of
submultiplicative weights.
\end{rem}

We use the following notation. For a complex manifold $M$ and   a
locally bounded function $\up\!:M\to [1,+\infty)$ denote by
$V_\up$  the closed absolutely convex hull of
$$
\{\up(x)^{-1}\de_x:\,x\in M\}
$$
in $\mathscr{A}(M)\!:=\cO(M)'$. It is noted in
\cite[Sect.~3.4.3]{Ak08} that $V_\up$ is a neighbourhood of~$0$ in
$\mathscr{A}(M)$; so $V_\up$ is an absorbent set. Therefore its
Minkowski functional is well defined on $\mathscr{A}(M)$; we
denote it by $\|\cdot\|_{\up}$. Let  ${\mathscr{A}}_\up(M)$ be the
completion of ${\mathscr{A}}(M)$ w.r.t. $\|\cdot\|_{\up}$. Also,
denote by ${\mathscr{A}}_{\up^\infty}(M)$ the completion of
${\mathscr{A}}(M)$ w.r.t. the sequence of prenorms
$(\|\cdot\|_{\up^n};\,n\in\N)$, where $\up^n(x)\!:=\up(x)^n$.

For any submultiplicative weight $\om$ on a complex Lie group $G$
the prenorm~$\|\cdot\|_{\om}$ is  submultiplicative and continuous
on $\mathscr{A}(G)$ [ibid., Lem.~5.1(a)]. (Note that main results
in [ibid.] is formulated for  Stein groups but their proofs work
for all complex Lie groups.) Thus ${\mathscr{A}}_\om(G)$ is a
unital Banach algebra, and the natural map $\mathscr{A}(G)\to
{\mathscr{A}}_\om(G)$ is a continuous homomorphism. Obviously, the
maximum of two submultiplicative weights is a submultiplicative
weight, so we have a directed projective system of unital Banach
algebras
$$
({\mathscr{A}}_\om(G):\,\text{$\om$ is a submultiplicative weight on
$G$})
$$ with natural connecting homomorphisms. Note that this
system is not empty because the trivial character $g\mapsto 1$ is
a submultiplicative weight. The following result is a
reformulation  of  [ibid., Th.~5.2(a)].

\begin{thm}\label{AMesw}
If $G$ is a complex Lie group, then the continuous homomorphism
$$\mathscr{A}(G)\to \varprojlim_\om  {\mathscr{A}}_\om(G)$$ is an
Arens-Michael envelope.
\end{thm}

Let $U$ be a  generating set for  $G$, i.e., $e\in U$ and
$\bigcup_{n=0}^{\infty} U^{n} = G$, where $U^{0}\!:=\{e\}$.
Recall that a locally compact group $G$ is called \emph{compactly
generated} if there is a relatively compact generating set $U$.
For given $U$, we define a  function $\ell_U$ on $G$ by
\begin{equation}\label{wordlen}
\ell_U(g)\!: = \min \{ n \!: \, g \in U^{n} \}\,.
\end{equation}
It is easy to see that $\ell_U$ is a length function, it is called
a \emph{word length function} (cf., e.g.,
\cite[Exm.~1.1.7]{Sch93}).

For given non-negative functions  $\tau_1$ and $\tau_2$ on a set
$X$ we say that $\tau_1$   \emph{dominated by}  $\tau_2$ (at
infinity) if there are $C,D>0$ s.t.
$$\tau_1(x)\le C\tau_2(x) + D\qquad (x\in X)\,.$$
Non-negative functions  $\tau_1$ and $\tau_2$ on a set $X$ are
said to be \emph{equivalent} (at infinity) if $\tau_1$ dominated
by $\tau_2$   and $\tau_2$   dominated by $\tau_1$.

\begin{pr} \label{decprli}
\emph{(cf. \cite[Th.~5.3]{Ak08})} Let $G$ be a compactly generated
complex Lie group, and let $U$ be a relatively compact generating
set. Put
$$
\xi(g)\!:=e^{\ell_U(g)} \,,
$$
where $\ell_U$ is the word length function defined by~\eqref{wordlen}. Then
the sequence of submultiplicative prenorms $(\|\cdot\|_{\xi^n};\,n\in \N)$
determines the topology on  $\widehat{\mathscr{A}}(G)$, i.e.,
$$
\widehat{\mathscr{A}}(G)\cong  {\mathscr{A}}_{\xi^\infty}(G)\,.$$
\end{pr}
\begin{proof}
It is easy to show (see \cite[Th.~1.1.21]{Sch93} or
\cite[Th.~5.3]{Ak08}) that every length function on $G$ is
dominated  by $\ell_U$.  Therefore for each submultiplicative
weight $\om$ on $G$ there are $C>0$ and $n\in\N$ s.t.
$\|\cdot\|_\om\le C\|\cdot\|_{\xi^n}$.
\end{proof}

\subsection*{Holomorphic functions of exponential type}

For  a complex Lie group $G$, denote by $\cO_{exp}(G)$ the linear
subspace of $\cO(G)$ that contains every function $f$ s.t. there
is a submultiplicative weight $\om$ satisfying  $|f(g)|\le \om(g)$
for all $g\in G $. A~holomorphic function $f$ on $G$ is of
\emph{exponential type}, if  $f\in\cO_{exp}(G)$
\cite[Sect.~5.3.1]{Ak08}.

To define the topology on $\cO_{exp}(G)$ we use the following notation. For
a complex manifold $M$ and  a locally bounded function $\up\!:M\to
[1,+\infty)$ denote by $\cO_\up(M)$ the linear subspace of $\cO(M)$ defined
by
\begin{equation} \label{fupn}
\cO_\up(M)\!:=\Bigl\{ f\in\cO(M) \!: |f|_\up\!:=\sup_{x\in
M}{\up(x)}^{-1}{|f(x)|}<\infty\Bigr\}\,.
\end{equation}
It is easy to see that  $\cO_\up(M)$ is a Banach space w.r.t
$|\cdot|_\up$.  Put $\cO_{\up^\infty}(M)\!:=\bigcup_{n\in\N}
\cO_{\up^n}(M)$. We consider $\cO_{\up^\infty}(M)$ with the
inductive limit topology.

For a complex Lie group $G$, we have an inductive system of
Banach spaces
$$(\cO_\om(G):\,\text{$\om$ is a submultiplicative weight on
$G$})$$ with natural connecting homomorphisms. Note that
$\cO_{exp}(G)=\bigcup_{\om} \cO_\om(G)$. So we can consider $\cO_{exp}(G)$
as a locally convex space via  identification
$$
\cO_{exp}(G)=\varinjlim_{\om} \cO_\om(G).
$$

\begin{pr} \label{decinli}\emph{(cf. \cite[Th.~5.3]{Ak08})}
Let $G$ be a compactly generated complex Lie group,  $U$  a
relatively compact generating set, and $\xi$  defined as in
Proposition~\ref{decprli}. Then
$$
\cO_{exp}(G)=  \cO_{\xi^\infty}(G)
$$
as locally convex spaces.
\end{pr}
\begin{proof}
Note again that  each length function on $G$ is dominated  by $\ell_U$.
\end{proof}
 \begin{rem}
In~\cite{Ak08}, Akbarov introduces  $\cO_{exp}(G)$  as an
inductive limit of the system $(\cO_{\om}(G))$, where each
$\cO_{\om}(G)$ is endowed with the topology of a Smith space, and
includes his results in more general context. In particular, if
$E$  is a stereotype locally convex space, he considers the dual
space $E^\bigstar$ endowed the topology of uniform convergence on
totally bounded subsets. Since each relatively compact subset of a
complete metric space is totally bounded  and each bounded subset
of a nuclear space is relatively compact \cite[\S~III.7.2,
Cor.~2]{SM}, we have $E^\bigstar=E'$ for every nuclear Fr\'echet
space $E$. So far as $\cO_{exp}(G)$ is dual to the nuclear
Fr\'echet space $\widehat{\mathscr{A}}(G)$ (see
Proposition~\ref{AOdul} below), our and Akbarov's approaches to
topology are equivalent. In addition, note that there is an
alternative proof of Proposition~\ref{AOdul}, which based
on~\cite[Th.1.11]{Ak08}.
\end{rem}

\begin{lm}\label{AupOup}
Let $M$ be a complex manifold.  For given locally bounded
function  $\up\!:M\to [1,+\infty)$, the pairing between
${\mathscr{A}}(M)$  and $\cO(M)$  induces the pairing  that makes
$\cO_\up(M)$ the dual Banach space to ${\mathscr{A}}_\up(M)$.
\end{lm}
\begin{proof}
Denote by $\langle \cdot,\cdot\rangle$ the pairing between
${\mathscr{A}}(M)$  and $\cO(M)$. Evidently,
$$S_\up\!:=\{f\in \cO(M):\,|f(x)|\le \up(x)\, \forall x \in M\}$$
is the unit ball in  $\cO_\up(M)$.   By \cite[Lem.~3.1]{Ak08}, the
set $S_\up$ is the polar of $V_\up$ (the closed  absolutely convex
hull of $\{\up(x)^{-1}\de_x:\,x\in M\}$) in ${\mathscr{A}}(M)$.
Therefore,
$$
|\langle \mu,f \rangle|\le \|\mu\|_\up\,|f|_\up\qquad (\mu\in
{\mathscr{A}}_\up(M),\,f\in\cO_\up(M))\,.
$$
So we have a bounded linear operator $\al\!:\cO(M)_\up\to
{\mathscr{A}}_\up(M)'$.

On the other hand,  suppose that $h\in {\mathscr{A}}_\up(M)'$. Let
$\rho_\up\!:{\mathscr{A}}(M)\to {\mathscr{A}}_\up(M)$ be the
completion map.  Then $h\rho_\up\in {\mathscr{A}}(M)'$ and the
corresponding function in $\cO(M)$ is determined by $x\mapsto
h\rho_\up(\de_x)$. Denote this function by $f(x)$. Since $h$ is
bounded,  there is $C>0$ s.t. $|f(x)|\le C\|\de_x\|_\up\le
C\up(x)$ for all $x\in M$, i.e., $f\in \cO_\up(M)$ and $|f|_\up\le
C$.   So we have a bounded linear operator
${\mathscr{A}}_\up(M)'\to\cO(M)_\up $, which we denote by $\be$.

It is obvious that $\be\al=1$. By the definition of
$\|\cdot\|_\up$, the linear span of $\{\de_x:\,x\in M\}$ is dense
in ${\mathscr{A}}_\up(M)$, so the operator $\be$ is injective.
Therefore we have a topological isomorphism.
\end{proof}

\begin{lm} \label{AOdul0}
Let $M$ be a complex manifold.  For given locally bounded
function  $\up\!:M\to [1,+\infty)$ and any  $n\in\N$ consider the
pairing $\langle \cdot,\cdot \rangle_{\up^n}$ between
${\mathscr{A}}(M)_{\up^n}$  and $\cO(M)_{\up^n}$ from
Lemma~\ref{AupOup}. If the Fr\'{e}chet space
$\mathscr{A}_{\up^\infty}(M)$ is reflexive, then there is a
paring between $\mathscr{A}_{\up^\infty}(M)$  and
$\cO_{\up^\infty}(M)$ that is compatible with all $\langle
\cdot,\cdot \rangle_{\up^n}$ and making
${\mathscr{A}}(M)_{\up^\infty}$  and $\cO(M)_{\up^\infty}$ strong
dual spaces to each other.
\end{lm}
\begin{proof}
Every reflexive Fr\'{e}chet space is distinguished, i.e., the
strong  dual is barreled \cite[\S 23.7]{Kot1}. For every
representation $E = \varprojlim E_n$ of a distinguished
Fr\'{e}chet space $E$ as a reduced (each $E\to E_n$ has dense
range) projective limit of a sequence of Banach spaces, its strong
dual space $E'$ equals $\varinjlim E_n'$ (see, e.g.,
 \cite{BDi}, Introduction).    So Lemma~\ref{AupOup} implies that
$\mathscr{A}_{\up^\infty}(G)'\cong \cO_{\up^\infty}(G)$. Finally,
reflexivity implies  $\cO_{\up^\infty}(G)'\cong
\mathscr{A}_{\up^\infty}(G)$.
\end{proof}

\begin{pr} \label{AOdul}
Let $G$ be a compactly generated complex Lie group.  Then there is a paring
between    $\widehat{\mathscr{A}}(G)$  and $\cO_{exp}(G)$ that is compatible
with all $\langle \cdot,\cdot \rangle_\om$, where $\om$ is a
submultiplicative weight, and making  $\widehat{\mathscr{A}}(G)$  and
$\cO_{exp}(G)$ strong dual spaces to each other.
\end{pr}
\begin{proof}
It follows from Propositions~\ref{decprli}  and~\ref{decinli} that
$$
\widehat{\mathscr{A}}(G)\cong
{\mathscr{A}}_{\xi^\infty}(G)\quad{\text{and}}\quad \cO_{exp}(G)\cong
\cO_{\xi^\infty}(G)\,.
$$
By \cite[Ths.~5.10, 6.2]{Ak08}, $\widehat{\mathscr{A}}(G)$ is
nuclear. Therefore it is reflexive and we can apply
Lemma~\ref{AOdul0}.
\end{proof}

Denote by ${\mathscr P}(G)$ the algebra of polynomial functions
on a simply connected complex Lie group $G $ w.r.t. the canonical
coordinates of the first kind (i.e., via the identification
$\exp\!:\fg\to G$) and consider the natural
paring
\begin{equation}\label{ufPp}
 \langle X,f\rangle\!:= Xf(e)\qquad (X\in U(\fg),\,
f\in {\mathscr P}(G)).
\end{equation}

\begin{co} \label{paUwhOe}
Let $G$ be a simply connected complex Lie group with Lie
algebra~$\fg$. Then the locally convex spaces $\widehat{U}(\fg)$
and $\cO_{exp}(G)$ are strong dual spaces to each other w.r.t.
the pairing extending the pairing between $U(\fg)$ and ${\mathscr
P}(G)$.
\end{co}
\begin{proof}
Obviously $\langle X,f\rangle= \langle \de_e, Xf\rangle$, so the
homomorphism $\tau\!:U(\fg) \to \mathscr{A}(G)$ defined
in~\eqref{taudef} induces a pairing between $U(\fg)$ and $\cO(G)$
that is continuously extended to $\widehat{U}(\fg)$. Finally, note
that each simply connected complex Lie group is compactly
generated \cite[Th.~7.4]{HeRo} and apply
Propositions~\ref{csicHFG} and~\ref{AOdul}.
\end{proof}

\section{Growth of  word length functions}\label{RM}
Let $\fg$ be a nilpotent complex or real Lie algebra and
$\mathscr{F}$  the lower central series of~$\fg$, which is
defined  by $\fg_i\!:=[\fg, \fg_{i-1}]$.   Fix an
$\mathscr{F}$-basis $e_1,\ldots, e_m$  in~$\fg$. For the simply
connected Lie group $G$  associated with~$\fg$ consider the
canonical coordinates of the first kind  $(t_1,\ldots, t_ m)$
and the canonical coordinates  of the second kind
$(\bar{t}_1,\ldots, \bar{t}_ m)$ determined by~$(e_i)$, i.e.,
$$
g=\exp\Bigl(\sum_{i=1}^m  t_i e_i\Bigr)=\prod_{i=1}^m \exp(\bar
t_i\,e_i) \qquad(g\in G)\,.
$$
Remind the definition of $w_i$ from~\eqref{widef} and set
\begin{equation}
\label{sidef}
 \si(g):=\max_i|t_i|^{1/w_i}\,,\qquad \bar\si(g):=\max_i|\bar{t}_i|^{1/w_i}\,.
\end{equation}
In  \cite{Go79} and \cite{Pir_stbflat} $\si$ is denoted by
$|\cdot|$ and is called 'homogeneous norm'.

The following theorem is the heart of our argument. See the proof
in Appendix.

\begin{thm}\label{eelUsi}
Let $G$ be a simply connected nilpotent (complex or real) Lie
group, and let $\ell$ be a word length function corresponding to
a relatively compact generating set. Then $\ell$, $\si$ and
$\bar\si$  are equivalent (at infinity).
\end{thm}

Now we can get the following explicit description of
$\cO_{exp}(G)$.

\begin{thm} \label{exptypdesc}
Let $G$ be a simply connected nilpotent complex Lie group with
Lie algebra $\fg$, and let  $(t_1,\ldots, t_ m)$ be the canonical
coordinates of the first kind associated with an
$\mathscr{F}$-basis in $\fg$, where $\mathscr{F}$ is the lower
central series. Then
\begin{multline*}
\cO_{exp}(G)= \bigl\{f\in \cO(G):\, \\ \exists C>0,\, \exists
r\in \R_+ \, \text{s.t.}\,|f(t_1,\ldots, t_ m)|\le C
e^{r\max_i|t_i|^{1/w_i}}\,\forall t_1,\ldots, t_ m \bigr\}
\end{multline*}
and  we have
$$
\cO_{exp}(G)\cong \varinjlim_{r\in \R_+} \cO_{\eta^r}(G)
$$
as locally convex spaces, where $\eta(t_1,\ldots, t_
m)\!:=e^{\max_i|t_i|^{1/w_i}}$ and the Banach space $\cO_{\eta^r}(G)$ is
defined  as in~\eqref{fupn}.

Moreover, we can replace $(t_1,\ldots, t_ m)$ by the canonical
coordinates of the second kind.
\end{thm}
We need the following lemma.
\begin{lm}\label{eqindlim}
Let $\tau_1$ and $\tau_2$ be non-negative locally bounded
functions on a complex manifold~$M$. If $\tau_1$ and $\tau_2$ are
equivalent  and  $\up_i(x)\!:=e^{\tau_i(x)}$ $(i=1,2)$, then
$\cO_{\up_1^\infty}(M) = \cO_{\up_2^\infty}(M)$ as subset of
$\cO(M)$ and as a locally convex space.
\end{lm}
\begin{proof}
For each $p\in \N$ there is $q\in \N$ s.t.
$\cO_{\up_1^q}(M)\subset \cO_{\up_2^p}(M)$ and this inclusion is a
 continuous linear map. Therefore we have a continuous linear map
$\cO_{\up_1^\infty}(M)\to\cO_{\up_2^\infty}(M)$ of inductive
limits. Similarly, we get a continuous linear map in the reverse
direction. Evidently, these maps are inverse to each other and
this completes the proof.
\end{proof}

\begin{proof}[Proof of Theorem~\ref{exptypdesc}]
Note that $G$ is compactly generated and fix a relatively compact
generating set~$U$. Proposition~\ref{decinli} implies that
$\cO_{exp}(G)= \cO_{\xi^\infty}(G)$, where $\xi(g)=e^{\ell_U(g)}$.
Evidently, $ \eta^r(g)\!:=e^{r\si(g)}$ ($r\in\R_+$). It follows
from Theorem~\ref{eelUsi} that $\si$  is equivalent to $\ell_U$;
therefore $\cO_{\xi^\infty}(M) = \cO_{\eta^\infty}(M)$ by
Lemma~\ref{eqindlim}. Obviously,
$$
\varinjlim_{r\in \R_+} \cO_{\eta^r}(G)=\varinjlim_{n\in\N}
\cO_{\eta^n}(G)\,,
$$
so the statement for $\si$ is proved.

Exactly the same argument is applied to $\bar \si$.
\end{proof}

\begin{rem}
In terms of the classical entire functions theory,
Theorem~\ref{exptypdesc} says, in particular, that $\cO_{exp}(G)$
consists of entire functions in $t_1,\ldots,t_m$ that are of at
most order $1/w_i$ and finite type w.r.t. $t_i$ for all
$i=1,\ldots,m$.
\end{rem}

\begin{exm} \cite[Sect.~5.4.2]{Ak08}
Let $\fg$ be an abelian  Lie algebra $\fg$ with basis
$e_1,\ldots,e_m$. Then $w_1=\cdots=w_m=1$ and
$\si(g)=\max\{|t_1|,\ldots,|t_m|\}$.  So $\cO_{exp}(G)$ coincides
with the set of entire functions  of exponential type (=of at most
order $1$ and finite type) on $\CC^m$ as it is defined in the
classical theory of several complex variables \cite[Def.~1.8]{LG}.
This example justifies the terminology.
\end{exm}

\begin{exm}
Consider the $3$-dimensional complex Heisenberg Lie algebra $\fg$ with basis
$e_1,e_2,e_2$  and relation $[e_1,e_2]=e_3$.  The lower central series  has
the form
$$
\fg=\fg_1 \supset\fg_2=\spn\{e_3\} \supset\fg_3=0\,.
$$
Then $w_1=w_2=1$, $w_3=2$, and
$\si(g)=\max\bigl\{|t_1|,\,|t_2|,\,|t_3|^{1/2}\bigr\}$.
\end{exm}

 \begin{exm}
 Let $\fg$ be the $7$-dimensional Lie algebra with
basis $e_1,\ldots ,e_7$ and commutation relations
\begin{gather*}
[e_1,e_i]=e_{i+1}\quad (i=2,\ldots ,6),\\
[e_2,e_3]=-e_6,\; [e_3,e_4]=e_7,\; [e_2,e_4]=[e_2,e_5]=-e_7\,,
\end{gather*}
the  undefined brackets being zero (see \cite{Favre} or
\cite[Ch.~2, Exm.~III.3(i)]{GoKh}).  Then
$$
w_1=w_2=1,\; w_3=2,\; w_4=3, \; w_5=4,\;w_6=5,\; w_7=6\,,
$$ and
$$
\si(g)=\max\bigl\{|t_1|,\,|t_2|,\,|t_3|^{1/2},\,|t_4|^{1/3},\,|t_5|^{1/4},\,|t_6|^{1/5},\,|t_7|^{1/6}\bigr\}\,.
$$

This algebra is exhibited  in \cite[Rem.~6.6]{Pir_stbflat} as an
example of a nilpotent Lie algebra that is not contractible and,
hence, is not positively graded. Thus, it does not satisfy
conditions of [ibid., Th.~6.19], which asserts that
$\io\!:U(\fg)\to \widehat U(\fg)$  is a   homological epimorphism
for any positively graded Lie algebra $\fg$. So our
Theorem~\ref{thmmail} is stronger that this Pirkovskii's result.
\end{exm}

\begin{rem}
Since we choose as $\mathscr{F}$  the lower central series
of~$\fg$, the dimensions of $\fg_1,\ldots,\fg_k$ are invariants of
a nilpotent Lie algebra~$\fg$. Therefore the sequence
$w_1,\ldots,w_m$ is also an invariant of~$\fg$. Of course, this
invariant (as well as the isomorphism class of the $\ptn$-algebra
$\cO_{exp}(G)$) is far from sufficient to distinguish  nilpotent
Lie algebras.  To do this one need the comultiplication  on
$\cO_{exp}(G)$ inherited from the multiplication on $U(\fg)$.
\end{rem}

\section{Proofs of  main results and an application}\label{PMR}

With Theorem~\ref{exptypdesc} under arms we can prove
Theorems~\ref{AMEdescrnil} and~\ref{thmmail}.

We need the following definition \cite[Def.~2.1]{Go79}: an
$\mathscr F$-weight sequence $\mathscr{M}$  is  \emph{entire} if
the following two conditions are satisfied
\begin{gather}
\label{entire}
\sum_\alpha M_\alpha r^{w(\alpha)} <\infty\quad\text{for all } r>0\,;\\
\sup_{\alpha,\beta\ne 0} \bigl\{ A^{w(\alpha)/w(\beta)}
M_\beta^{1/w(\beta)} M_\alpha^{-1/w(\alpha)}\bigr\} <\infty
\quad\text{for some } A>0\,.\notag
\end{gather}

\begin{proof}[Proof of Theorem~\ref{AMEdescrnil}]
Let $G$ be a simply connected nilpotent complex Lie group with Lie
algebra $\fg$. Given $z\in\CC$, define a linear map
$\delta_z\!:\fg\to\fg$ by $\delta_z(e_i)=z^{w_i}e_i$. We use the
same symbol $\delta_z$ to denote the corresponding holomorphic
endomorphism of $G$ satisfying $\delta_z\circ\exp=
\exp\circ\delta_z$.

Recall that our choice of an $\mathscr F$-weight sequence is
$\mathscr M_1$ defined by $ M_\al\!:=w(\al)^{-w(\al)}$. Consider
the growth  function $\Phi$ associated with $\mathscr{M}_1$ given
by
$$
\Phi(g)\!:=\sum_\alpha M_\alpha \si(g)^{w(\alpha)}=\sum_\alpha
\left(\frac{\si(g)}{w(\al)}\right)^{w(\alpha)}\,,
$$
where the function $\si$ is defined in~\eqref{sidef}.   (Since
$\mathscr{M}_1$ is entire \cite[Sect.~2, Exm.~1]{Go79}, the
function $\Phi$ is well defined.) Denote by
$\cO_{\mathscr{M}_1}(G)$ the linear subspace of $\cO(G)$ that
contains every function $f$ s.t. there are $C>0$ and $r>0$ and
$|f(g)|\le C\Phi(\de_r g)$ is satisfied for all $g\in G $. To
make $\cO_{\mathscr{M}_1}(G)$ a locally convex space consider,
for $r>0$,  the space
$$
\cO_{{\mathscr{M}_1},r}(G)\!:=\Bigl\{ f\in\cO(G) : N_r(f)\!:=\sup_{g\in
G}\Phi(\de_r g)^{-1}|f(g)|<\infty\Bigr\}\,.
$$
Evidently, $\cO_{{\mathscr{M}_1},r}(G)$ is a Banach space w.r.t.
the norm $N_r$. Since $\Phi(\delta_s g)\le \Phi(\delta_r g)$
whenever $0\le s\le r$, we have $\cO_{{\mathscr{M}_1},s}(G)\subset
\cO_{{\mathscr{M}_1},r}(G)$ for each $s\le r$, and $N_r(f)\le
N_s(f)$ for each $f\in \cO_{{\mathscr M}_1,s}(G)$. Therefore one
may identify $\cO_{\mathscr{M}_1}(G)$ with the locally convex
space $ \displaystyle\varinjlim_{s\in\R_+}
\cO_{{\mathscr{M}_1},s}(G)$.

On the other hand,  by Theorem~\ref{exptypdesc},
$\cO_{exp}(G)\cong \displaystyle{\varinjlim_{r\in\R_+}
\cO_{\eta^r}(G)} $, where $\eta(g)\!:=e^{\si(g)}$. Goodman noted
in [ibid., (2.6)] that there are positive constants $c,a,C,A$ s.t.
$$
ce^{a\si(g)}\le \Phi(g) \le Ce^{A\si(g)}\qquad(g\in G)\,.
$$
Applying the same argument as in the proof of Lemma~\ref{eqindlim}
to the inductive limits $ \displaystyle\varinjlim_{s\in\R_+}
\cO_{{\mathscr{M}_1},s}(G)$ and
$\displaystyle{\varinjlim_{r\in\R_+} \cO_{\eta^r}(G)}$ we have
from this observation that $\cO_{\mathscr{M}_1}(G)=\cO_{exp}(G)$
as a set and as a locally convex space.

It is shown in [ibid., Th.~3.1] that  $U(\fg)_{\mathscr{M}_1}$ is
the strong  dual  space of $\cO_{\mathscr{M}_1}(G)$ via the
pairing that extends the pairing between $U(\fg)$ and ${\mathscr
P}(G)$ defined in~\eqref{ufPp}\footnote{Note that to prove this
result Goodman used an alternative definition of
$U(\fg)_{\mathscr{M}}$ for a $\mathscr F$-weight sequence. In
contrast to \eqref{UfgMde}, which introduced in~\cite{Go78}, he
considered in~\cite{Go79} series in symmetrization of the
PBW-basis $(e^\al\,:\al\in\Z_+^m)$. Nevertheless, it follows from
\cite[Lem,~6.2 and Rem.~(1) on p.~203]{Go78}  that the approaches
are equivalent.}. On the other hand, by Corollary~\ref{paUwhOe},
$\widehat{U}(\fg)$ is the strong  dual  space of  $\cO_{exp}(G)$
via the same paring. Thus $U(\fg)_{\mathscr{M}_1}$ and
$\widehat{U}(\fg)$  are isomorphic as $\ptn$-algebras, and
$U(\fg)\to U(\fg)_{\mathscr{M}_1}$ is an Arens-Michael envelope.
\end{proof}

\begin{proof}[Proof of Theorem~\ref{thmmail}]
Since $\mathscr{M}_1$ is entire \cite[Sect.~2, Exm.~1]{Go79}, we
can refer to Pirkovskii's Theorem~\ref{PirkThUM}, which asserts
that, in this case, the natural map  $U(\fg)\to U(\fg)_{\mathscr
M_1}$ is a   homological epimorphism.  Thus, the assertion
follows Theorem~\ref{AMEdescrnil}.
\end{proof}

As an application of Theorem~\ref{AMEdescrnil} we obtain an
estimation of a submultiplicative norm for powers of elements in
a nilpotent complex Lie algebra~$\fg$. Let $ \mathscr{F}$ denote,
as usual, the lower central series.   Given $X\in\fg$ s.t. $X\ne
0$, set $w(X)\!:=\max\{ j: X\in\fg_j\}$.

If $\|\cdot\|$ is a submultiplicative prenorm on $U(\fg)$, then
denote by $A$ the completion of $U(\fg)$ w.r.t. $\|\cdot\|$ and
by $\la\!:\fg\to A$ the corresponding Lie algebra homomorphism.
It is easy to see from the spectral properties of  Banach
algebras that, for any $X\in [\fg,\fg]$, the element $\la(X)$  is
topologically nilpotent, i.e., $\|\la(X)^n\|^{1/n}=o(1)$. (In
fact, $\la([\fg,\fg])$ is contained in the Jacobson radical of
$A$~\cite{Tu84}.)   Theorem~\ref{AMEdescrnil} gives us  the
following decay  estimation.

\begin{pr}
Let $A$ be a unital Banach algebra with a submultiplicative norm
$\|\cdot\|$, let $\fg$ be a nilpotent complex Lie algebra with
lower central series $\mathscr{F}$ as a positive filtration, and
let $\la\!:\fg\to A$ be a Lie algebra homomorphism. Then for each
$X\in\fg\setminus\{0\}$,
$$
\|\la(X)^n\|^{1/n}=O\biggl(\frac{1}{n^{w(X)-1}}\biggr)\qquad
(n\in \N)\,.
$$
\end{pr}
Note that in the case when $w(X)=1$ we have the trivial assertion
that $\|\la(X)^n\|^{1/n}$ is bounded but for $w(X)>1$ the
statement is more interesting.

\begin{proof}
Consider the system of prenorms $(\|\cdot\|_r;\, r>0)$ on
$U(\fg)$ from Theorem~\ref{AMEdescrnil}. By the universal
property, on can extend $\la$ to a homomorphism $U(\fg)\to A$.
Then $\|\la(\cdot)\|$ is a submultiplicative prenorm on $U(\fg)$.
It follows from Theorem~\ref{AMEdescrnil} that $\|\la(\cdot)\|$
is continuous w.r.t. $(\|\cdot\|_r;\, r>0)$. Thus it is
sufficient to show  that $\|X^n\|_r^{1/n}=O(n^{1-w(X)})$  for
every $X\ne0$ and every $r>0$.

Fix $X\in\fg\setminus\{0\}$. We claim that there is  an
$\mathscr{F}$-basis  $(e_i)$ in $\fg$ s.t. $X=e_i$ for some~$i$.
Indeed, let $j\!:=w(X)$. Then $X\in \fg_j$ and $X\not\in
\fg_{j+1}$. Consider an arbitrary $\mathscr{F}$-basis  $(e_i)$ in
$\fg$. Since       $\fg_j = \spn\{e_i : w_i \ge j\}$, we have
$X=\sum_{w_i \ge j}c_ie_i$ for some $c_i\in\CC$ s.t. there is $i$
with $w_i = j$ and $c_i\ne 0$. So we can replace $e_i$ by $X$ and
get an $\mathscr{F}$-basis again.

If $\al\!:=(0,\ldots,n,\ldots)$, where $n$ is on the $i$th place, then
$w(\alpha)=nw_i$ and $\al!=n!$. Therefore,
$$
\|e_i^n\|_r=n!\left(\frac{r}{nw_i}\right)^{nw_i}\,.
$$
Hence, $\|e_i^n\|^{1/n}=O(n^{1-w_i})$. Finally, remind that
$e_i=X$ with $w_i=w(X)$.
\end{proof}

\section{Appendix. Proof of Theorem~\ref{eelUsi} and relations with Riemannian distances}\label{AppC1}

The case of Theorem~\ref{eelUsi} when $\fg$ is complex  is easily
reduced for the real case, so below  we suppose that $\fg$ is a
nilpotent real Lie algebra.

The argument consists of three parts and we need three lemmas.
Both statements of the first lemma are partial cases of
\cite[Lem.~IV.5.1]{VSC92}.
\begin{lm}\label{fromV}
\emph{(A)} For each $n\in\N$  there are $N\in\N$, $s_1,\ldots,
s_{N}\in \{1,\ldots,n\}$, and $\al_1,\ldots, \al_{N}\in \R$ s.t.
for arbitrary $Y_1,\ldots, Y_n\in\fg$
\begin{equation}\label{exppro}
\exp\Bigl(\sum_{s=1}^n Y_s\Bigr)=\prod_{r=1}^N \exp(\al_r
Y_{s_r})\,.
\end{equation}

\emph{(B)} For each $n\in\N$ and each word $U=(u_1,\ldots,u_j)$
in $\{1,\ldots,n\}$ there are $N'\in\N$, $s'_1,\ldots, s'_{N'}\in
\{1,\ldots,n\}$, and $\al'_1,\ldots, \al'_{N'}\in \R$ s.t. for
arbitrary $Y_1,\ldots, Y_n\in\fg$
\begin{equation}\label{exppro2}
\exp(Y_U)=\prod_{p=1}^{N'} \exp(\al'_p Y_{s'_p})\,,
\end{equation}
where  $Y_U\!:=[Y_{u_1},[Y_{u_2},\cdots, Y_{u_j}]\cdots]$.
\end{lm}

\begin{lm}\label{sisubpo}
There are $C,D\ge 0$ s.t.
\begin{equation}\label{sig12}
\si(g_1 g_2)\le C\max\{\si(g_1),\,\si(g_2)\}+D\qquad(g_1,g_2\in G)
\end{equation}
i.e., $e^\si$ is 'sub-polynomial' in the terminology of
\cite[(1.3.1)]{Sch93}.
\end{lm}
\begin{proof}
Let $k$  be the positive integer s.t. $\fg_k\ne 0$ and $\fg_{k+1}=
0$. For $X,Y\in\fg$ denote by $X\ast Y$ the Hausdorff product,
i.e., $\exp(X\ast Y)=\exp X\exp Y$. It follows from the
Baker-Campbell-Hausdorff formula that, for each  word $U$ in
$\{1,\ldots,2m\}$ of length at most $k$, there is $\be_U\in \R$
s.t.
\begin{equation}\label{YmY2m}
\Bigl(\sum_{i=1}^m Y_i\Bigr)\ast \Bigl(\sum_{i=m+1}^{2m}
Y_i\Bigr)=\sum_U \be_U Y_U
\end{equation}
for every $Y_1,\ldots,Y_{2m}\in \fg$.

Write $g_1=\exp(\sum_{i=1}^m t_ie_i)$ and $g_2=\exp(\sum_{i=1}^m
t_{m+i}e_i)$. Substituting in \eqref{YmY2m}, we obtain
$$
\Bigl(\sum_{i=1}^m t_ie_i\Bigr)\ast \Bigl(\sum_{i=1}^m
t_{m+i}e_i\Bigr)=\sum_U \be_U t_{u_1}\cdots t_{u_s}E_U\,,
$$
where $E_U=[e_{u_1},[e_{u_2},\cdots, e_{u_s}]\cdots]$. Denote
$u_i$ modulo~$n$ by~$\bar u_i$. Then $E_U\in \fg_{j(U)}$, where
$j(U)\!:=\sum_{i=1}^s w_{\bar u_i}$. So we have
$$
E_U=\sum_{w_p\ge j(U)} \al_{U,p} e_p
$$
for some $\al_{U,p}$.

Now write $g_1g_2=\exp(\sum_{i=1}^m t'_ie_i)$. Then the
coefficient $t'_p$ is bounded by
$$
\sum_{U}|\be_U|\,|\al_{U,p}|\,|t_{u_1}\cdots t_{u_s}|\,.
$$
Note that $w_p\ge j(U)$ implies
$$
|t_{u_1}\cdots t_{u_s}|^{1/w_p}\le |t_{u_1}\cdots
t_{u_s}|^{1/j(U)}+1\le \sum_{i=1}^s|t_{u_i}|^{1/w_{\bar
u_i}}+1\,.
$$
Hence there exist $C'$ and $D'$ s.t.
$$
\sum_{p=1}^m|t'_p|^{1/w_p}\le
C'\left(\sum_{i=1}^m|t_i|^{1/w_i}+\sum_{i=1}^m|t_{m+i}|^{1/w_i}\right)+D'\,,
$$
which implies \eqref{sig12}.
\end{proof}

Fix subspaces $\fv_1,\ldots,\fv_k$ s.t.
$\fv_i\oplus\fg_{i+1}=\fg_i$. Evidently,
$\fg=\oplus_{i=1}^k\fv_i$.  It can easily be checked that this
decomposition can assumed compatible with an $\mathscr{F}$-basis.
Set $\fv^{(1)}\!:=\fv_{1}$ and
$\fv^{(j)}\!:=[\fv_{1},\fv^{(j-1)}]$ for $j>1$. It is not hard to
see that
\begin{equation}\label{fgjk}
\fg_j=\fv^{(j)}+\cdots+ \fv^{(k)}\,.
\end{equation}

\begin{lm}\label{normtrick}
For any $C>0$ there is a norm $\|\cdot\|$ on $\fg$ (as a real
linear space) s.t. the following conditions are satisfied.
\begin{enumerate}
\item If $V=\sum_s V_s$, where $V_s\in \fv_s$, then $\|V_s\|\le
\|V\|$ for all $s$.

\item For any $p\in\N$, $u_1,\ldots,u_p\in \{1,\ldots,k\}$, and $V_{u_s}\in \fv_{u_s}$
$(s=1,\ldots,p)$, one has
$$
\|\,[V_{u_1},[V_{u_2},\cdots, V_{u_p}]\cdots]\,\|\le C
\|V_{u_1}\|\,\|V_{u_2}\|\cdots \|V_{u_p}\|\,.
$$
\end{enumerate}
\end{lm}
\begin{proof}
Fix a norm $\|\cdot\|_s$ on each $\fv_s$ and consider norms on
$\fg$ of the form $\|\sum_s V_s\|=\sum \la_s\|V_s\|_s$
$(\la_s>0)$. Obviously, any such norm satisfies to~(1).
Proceeding by induction on $p$ it is not hard to show that there
is a norm of this form that satisfies to~(2).
\end{proof}

\begin{proof}[Proof of Theorem~\ref{eelUsi} (the real case)]
Let $\ell$ and $\ell'$ be the word length functions corresponding
to the generating sets
$$
\bigcup_{i=1}^m\{\exp( t_ie_i)\!:\,|t_i|\le 1\}\quad\text{and}\quad\{\exp(X)\!:\, X\in\fv_1,\,\|X\|\le 1\}\,,
$$
resp. The first set is generating, since for each $i$ the linear
span of $\{e_i,\ldots,e_m\}$ is a subalgebra of $\fg$. To see that
the second set is generating one can apply \eqref{fgjk} with
$j=1$, the surjectivity of the exponential map, and both parts of
Lemma~\ref{fromV}. Since  $G$ is compactly generated, all word
length functions are equivalent \cite[Th.~1.1.21]{Sch93} (cf. also
Corollary~\ref{eUdl} below). Thus, it suffices to show that
$\ell\lesssim \bar\si \lesssim \si \lesssim\ell'$, where
$\lesssim$ means ''is dominated by''.

(1) First, we prove that $\ell\lesssim \bar\si$. Denote by $S$
the subset of indices $p$ s.t. $w_p=1$ (eq., $e_p\in\fv_1$). For
any word $U=(u_1,\ldots,u_j)$ in $S$ set $\la(U)=j$ and consider
the $j$th commutator $E_U(t)\!:=[te_{u_1},[te_{u_2},\cdots,
te_{u_j}]\cdots]$ for $t\in\R$.

Fix $j\in\{1,\ldots,k\}$ and $e_i$ with $w_i=j$. Since $e_i\in
\fg_j$, it follows from \eqref{fgjk} that $e_i$ is a linear
combination $\sum_U\mu_U E_U(1)$, where $U$ runs all words in $S$
with $j\le \la(U)\le k$. Therefore, for any $t\in\R$,
\begin{equation}\label{teinuU}
te_i=\sgn(t)\, \sum_{j\le \la(U)\le k}\mu_U
E_U\bigl(|t|^{1/\la(U)}\bigr) \,.
\end{equation}
For simplicity, we assume that $t>0$; for negative $t$ the
following estimates are the same.

Enumerating all words in the sum above as $U_1,\ldots,U_n$ and
applying Part~(A) of Lemma~\ref{fromV} for $n$, we can substitute
$\sum\mu_U E_U(|t|^{1/\la(U)})$ in \eqref{exppro} and get from
\eqref{teinuU} the equality
\begin{equation*}
\exp(te_i)=\prod_{r=1}^N \exp\bigl(\al_r
\mu_{U_r}E_{U_r}\bigl(|t|^{1/\la(U_r)}\bigr)\bigr)\,.
\end{equation*}
Further, applying Part~(B) of Lemma~\ref{fromV} to each factor in
the product with $Y_1=\al_r\mu_{U_r} |t|^{1/\la(U_r)} e_1$ and
$Y_s=|t|^{1/\la(U_r)} e_s$ for $s\ge 2$, we have  that there are
$N''$ and $\be_1,\ldots,\be_{N''}$ independent in $t$ with some
positive integers $\la_1,\ldots,\la_{N''}$ non less that~$j$ s.t.
\begin{equation*}
\exp(te_i)=\prod_{p=1}^{N''} \exp\bigl(|t|^{1/\la_p} \be_p
e_{i_p}\bigr)\,.
\end{equation*}
Hence, for each $t\ge 0$,
$$
\ell(\exp te_i)\le \sum_p \ell\bigl(\exp\bigr(|t|^{1/\la_p} \be_p
e_{i_p}\bigr)\bigr)\le \sum_p (|t|^{1/\la_p} \be_p+1) \,.
$$
Since $|t|^{1/\la_p}\le |t|^{1/j}$ for $|t|\ge 1$, we obtain
$\ell(\exp(te_i))\le C_j |t|^{1/j}+D_j$ for all $t$, where $C_j$
and $D_j$ depend only on $j$.

Finally, write any $g\in G$ as
$$
g=\prod_{i=1}^m \exp(\bar t_i\,e_i)\,,
$$
where $\bar{t}_1,\ldots, \bar{t}_ m\in \R$. Therefore
$$
\ell(g)\le \sum_i \ell(\exp(\bar t_i\,e_i))\le \sum_i C_{w_i}
|\bar t_i|^{1/{w_i}}+D_{w_i}\,.
$$
Thus, $\ell$ is dominated by $\bar\si$.

(2) Secondly, we prove that $\bar\si\lesssim\si$ (cf. the proof
of \cite[II.4.17]{DER}). We show by induction on $i$ in the
reverse order that for each $i=1,\ldots,m$ there are constants
$A_i$ and $B_i$ s.t. for every $g=\exp(\sum_{s=i}^m t_se_s)$ the
estimate $\bar\si(g)\le A_i \si(g)+B_i$ holds.

Obviously, if $i=m$, then $\bar\si(g)=\si(g)$. Suppose that the
induction assumption is satisfied for $i$.  The
Baker-Campbell-Hausdorff  formula implies that
\begin{equation}\label{BCHfi1}
 (-t_{i-1}e_{i-1})\ast\Bigl(\sum_{s=i-1}^m
t_se_s\Bigr)=\sum_{s=i}^m t'_se_s
\end{equation}
for some $t'_s$. Write $\exp(\sum_{s=i}^m t'_se_s)=\prod_{s=i}^m
\exp(\bar t_s\,e_s)$. Then
$$
\exp\Bigl(\sum_{s=i-1}^m t_se_s\Bigr)=\prod_{s=i-1}^m \exp(\bar
t_s\,e_s)\,,
$$
where $\bar t_{i-1}\!:=t_{i-1}$. Applying Lemma~\ref{sisubpo} to
\eqref{BCHfi1}, we get
$$
\max_{i\le s\le m}|t'_{s}|^{{1/w_s}}\le C\max_{i-1\le s\le
m}|t_{s}|^{{1/w_s}} +D\,.
$$
By the inductive assumption, we have
$$
\max_{i\le s\le m}|\bar{t}_{s}|^{{1/w_s}}\le A_i\max_{i\le s\le
m}|t'_{s}|^{{1/w_s}}+B_i\,.
$$
Combining the two estimates we obtain
$$
\max_{i-1\le s\le m}|\bar{t}_{s}|^{{1/w_s}}\le
A_{i-1}\max_{i-1\le s\le m}|t_{s}|^{{1/w_s}}+B_{i-1}
$$
for some $A_{i-1}$ and $B_{i-1}$ depending only on $i$. The
induction is complete.

Finally, note that we have shown that $\bar\si(g)\le A_1
\si(g)+B_1$ for all $g\in G$.

(3) Thirdly, we prove that $\si\lesssim\ell'$. Let $\|\cdot\|$ be
the norm on $\fg$ existing by Lemma~\ref{normtrick} (the value of
constant $C$ is specified below). Note that $\si(g)$ is equivalent
to the function $g\mapsto \max_j\|V_j\|^{1/j}$, where
$g=\exp(\sum_j V_j$) with $V_j\in\fv_j$. So it suffices to show
that $\ell'(g)=n$ implies $\|V_j\|\le n^j$ for all~$j$.

We proceed by induction. For $n=0$ and $n=1$ the claim is
obvious. Suppose that it holds for  $ n-1\ge 1$. If $\ell'(g)=n$,
then $g=g_1g_2$, where $\ell'(g_1)=1$ and $\ell'(g_2)= n-1$,
i.e., $g_1=\exp V_0$, where $V_0\in\fv_1$ with $\|V_0\|\le 1$ and
$g_2=\exp (\sum_j V_j)$, where $V_j\in\fv_j$ with $\|V_j\|\le
(n-1)^j$. Write $g_1g_2=\exp(\sum_{j=1}^k W_j)$, where
$W_j\in\fv_j$. We need to show that $\|W_j\|\le n^j$ for all $j$.

Note that, by the Baker-Campbell-Hausdorff  formula, there are
$\ga_U$ s.t.
\begin{equation}\label{corCBH}
V_0\ast(V_1+\cdots+V_K)=\exp\Bigl(V_0+V_1+\cdots+V_K+\sum_U \ga_U
V_U\Bigr)\,,
\end{equation}
for  any $V_0\in\fv_1$ and $V_j\in\fv_j$ ($j=1,\ldots,k$), where
$V_U\!:=[V_{u_1},[V_{u_2},\cdots, V_{u_p}]$ and
$U=(u_1,\ldots,u_p)$ runs all words in $\{0,1,\ldots,k\}$ of
length at least $2$ and at most $k$, and containing at least~$1$
of occurrence of~$0$.

Further, for each $U$, there is a unique decomposition
$$
V_U=\sum_{j=|U|}^k Y_{U,j}\,;\qquad(Y_{U,j}\in\fv_j)\,,
$$
where $|U|$ denotes the sum of $u_1+\cdots+u_p$ and the number of
occurrence of~$0$ in~$U$. We get from \eqref{corCBH} that
$W_1=V_0+V_1$ and for $j\ge 2$
\begin{equation}\label{WjVkS}
W_j=V_j+\sum_{U\in S_j} \ga_U Y_{U,j}\,,
\end{equation}
where $S_j$ denotes the set of words $U$ s.t. the length of $U$ is
at least $2$, $|U|\le j$, and $U$ contains at least~$1$ of
occurrence of~$0$.

By setting $C\!:= (\sum_{U} |\ga_U|)^{-1}$  in
Lemma~\ref{normtrick} we get
\begin{multline}\label{gaYUje}
\Bigl\|\sum_{U\in S_j} \ga_U Y_{U,j}\Bigr\|\le \sum_{U\in S_j} |\ga_U|\, \|Y_{U,j}\|\le\text{(by Part~(A))}\\
 \sum_{U\in S_j} |\ga_U|\,
\|V_{U}\|\le \text{(by Part~(B))} \\
\sum_{U\in S_j} |\ga_U|\, C \,\|V_{u_1}\|\,\|V_{u_2}\|\cdots \|V_{u_p}\| \le\\
\max_{U\in S_j}\{\|V_{u_1}\|\,\|V_{u_2}\|\cdots \|V_{u_p}\|\}\,.
\end{multline}
According to the inductive assumption, $\|V_{u_s}\|\le
(n-1)^{u_s}$ for all $u_s$. Since $U\in S_j$, we have
$u_1+\cdots+u_p< |U|\le j$. Therefore,
$$\|V_{u_1}\|\,\|V_{u_2}\|\cdots \|V_{u_p}\|\le (n-1)^{j-1}\,.$$
Finally, \eqref{WjVkS} and \eqref{gaYUje} imply $\|W_j\|\le
2(n-1)^{j-1}\le n^j$ for $j\ge 2$ and obviously $\|W_1\|\le
\|V_0\|+\|V_1\|\le n$.
\end{proof}

\subsection*{On left invariant Riemannian distances}
The following remarks explain why  Theorem~\ref{eelUsi} can be
reformulated in terms of left invariant (sub)-Riemannian
distances; thus, we have a direct connection with results
in~\cite{Be96,DER,Kar94,VSC92}.

We say that a  length function $\ell$ is \emph{symmetric} if
$\ell(e) = 0$ and $\ell(g^{-1})=\ell(g)$  for all $g\in G$.
Obviously, for given  length function $\ell$, the function $\ell'$
defined by
\begin{equation*}
\ell'(g)\!:=\max\{\ell(g),\,\ell(g^{-1})\},\quad(g\ne
e)\qquad\text{and}\quad \ell'(e)\!:=1
\end{equation*}
is a symmetric length function.

For a symmetric length function $\ell$ consider the following
condition:

$(\al)$ \emph{there exists $C$ such that, for each $g\in G$
satisfying $\ell(g) \ge 1$, there are $g_0,\ldots,g_n\in G$ s.t.
$g_0 = e$, $g_n = g$, and $\ell(g_i^{-1}g_{i+1}) \le 1$ $(i =
0,\ldots, n-1)$ with $n \le C\ell(g)$.}

It is trivial that the formulas ${ }_l d(g,h)\!:=\ell(g^{-1}h)$
and ${ }_d\ell(g)\!:=d(e,g)$ determine a bijection between the
set of symmetric length functions on $G$  and the set of locally
bounded left invariant distances on $G$. Under this
correspondence our condition ($\al$) becomes condition~$(C_3)$
from p.~40 of  \cite{VSC92}. In the terminology of
\cite[Def.~3.B.1]{CH16}, property~$(C_3)$ means that $d$ is
\emph{$1$-large-scale geodesic}.

\begin{pr}\label{slfeq}
Let $G$ be a locally compact group and let $\ell_1$ and $\ell_2$
be symmetric length functions on $G$ satisfying  to condition
$(\al)$. Suppose that $\ell_1$ is bounded on $\{g\in G:\,
\ell_2(g)\le 1\}$ and $\ell_2$ is bounded on $\{g\in G:\,
\ell_1(g)\le 1\}$. Then $\ell_1$ and $\ell_2$ are equivalent.
\end{pr}
\begin{proof}
It is noted in \cite[Rem.~III.4.3]{VSC92} that the argument from
[ibid., Prop.~III.4.2] can be applied in our situation.
\end{proof}

Recall that every connected locally compact group is compactly
generated \cite[Th.~7.4]{HeRo}.

\begin{pr}\label{eUdl}
Let $G$ be a connected real Lie group. Suppose that $U$ is a
symmetric relatively compact generating set and $d$ is a distance
determined by a left invariant Riemannian metric. Then $\ell_U$
is equivalent (at infinity) to~${
}_d\ell$.
\end{pr}
\begin{proof}
It is easy to see that $\ell_U$ is bounded on $\{g\in G:\, {
}_d\ell(g)\le 1\}$ and ${ }_d\ell$ is bounded on $\{g\in
G:\,\ell_U(g)\le 1\}$. By Proposition~\ref{slfeq}, it suffices to
show that $\ell_U$ and~${ }_d\ell$ satisfy to $(\al)$.

Let $\ell_U(g)=n>0$ and fix $g=h_1\cdots h_n$ where $h_i\in U$.
Set $g_0=e$, and   $g_i=h_1\cdots h_i$ for $i=1,\ldots,n$. Then
$\ell_U(g_i^{-1}g_{i+1})=\ell(h_{i+1}) \le 1$ for $i<n$. Thus
$(\al)$ is satisfied for $\ell_U$ with $C=1$. Since $d$ is
defined as the infimum of lengths of piecewise smooth paths,
condition $(\al)$ is satisfied for ${ }_d\ell$.
\end{proof}

For a nilpotent Lie group $G$, the distance  $d$ in
Proposition~\ref{eUdl} can be replaced by a Carnot-Carath\'eodory
distance in the sense
of~\cite[Sect.~III.4]{VSC92}.

\end{document}